\documentclass[preprint,11pt]{elsarticle}
\usepackage{subcaption}
\usepackage{amsmath,amsthm }
\usepackage{multirow}

\DeclareMathOperator*{\argmin}{arg\,min}
\usepackage{amsfonts}
\usepackage{graphics}
\usepackage[
letterpaper, top=1.2in, bottom=1.2in, left=1.2in, right=1.2in]{geometry}
\usepackage{xcolor}
\usepackage{amssymb}
\usepackage{psfrag}
\usepackage{graphicx}
\usepackage[utf8]{inputenc}
\usepackage{bm}
\usepackage{mathtools}
\usepackage{epstopdf}
\usepackage[colorlinks,linkcolor=black,anchorcolor=black,citecolor=black]{hyperref}
\usepackage{cleveref}
\usepackage{setspace}
\usepackage{amsmath,amssymb,bbm}
\usepackage[utf8]{inputenc}
\usepackage{url}
\usepackage{color}
\usepackage{booktabs}
\usepackage[linesnumbered,ruled]{algorithm2e}
\usepackage{placeins}
\newtheorem{remark}{Remark}
\usepackage[title]{appendix}
\usepackage{threeparttable}
\usepackage{nicematrix}
\usepackage{blkarray}

\DeclareMathAlphabet\mathbfcal{OMS}{cmsy}{b}{n}
\newtheorem{lemma}{Lemma}[section]

\makeatletter
\def\ps@pprintTitle{%
  \let\@oddhead\@empty
  \let\@evenhead\@empty
  \def\@oddfoot{\reset@font\hfil\thepage\hfil}
  \let\@evenfoot\@oddfoot
}
\makeatother

\begin{document}

\begin{frontmatter}

\title{Integration of discrete-event dynamics and machining dynamics for machine tool: modeling, analysis and algorithms}

\author[a]{Mason Ma} 
\author[a]{Alisa Ren}
\author[b]{Christopher Tyler}
\author[b]{Jaydeep Karandikar}
\author[c]{Michael Gomez}
\author[a]{Tony Shi\corref{cor1}}
\ead{tony.shi@utk.edu}
\author[b,d]{Tony Schmitz}

\address[a]{Department of Industrial and Systems Engineering, University of Tennessee Knoxville, Knoxville, TN 37996, USA}
\address[b]{Manufacturing Science Division, Oak Ridge National Laboratory, Oak Ridge, TN 37830, USA}
\address[c]{MSC Industrial Supply Company, Knoxville, TN 37932, USA}
\address[d]{Department of Mechanical, Aerospace and Biomedical Engineering, University of Tennessee Knoxville, Knoxville, TN 37996, USA}

\cortext[cor1]{Corresponding author. Tel.: +1-865-974-7654.}

\begin{abstract}
Machining dynamics research lays a solid foundation for machining operations by providing stable combinations of spindle speed and depth of cut. Furthermore, machine learning has been applied to predict tool life as a function of cutting speed. However, the existing research does not consider the discrete-event dynamics in machine shop, i.e., the machine tool needs to process a series of parts in queue under various practical production requirements.
This paper addresses the integration of discrete-event dynamics and machining dynamics to achieve cost savings in machining. 
We first propose a learning-based cost function for the studied integrated optimization problem of machine tool. 
The proposed cost function utilizes the predicted tool life under different stable cutting speeds for further optimizing speed selection of machine tool to deal with the discrete-event dynamics in machine shop. Then, according to the practical production requirements, we develop several mathematical optimization models for the related integrated optimization problems with the consideration of cost, makespan and due date. 
Numerical results show the effectiveness of our proposed methods and also the potential to be used in practice.
\end{abstract}


\end{frontmatter}

\section{Introduction}
\label{sec:intro}

The daily operations of machine tools aim to produce products to fulfill orders with various production requirements at the minimum cost.
This involves both the machining dynamics that govern the material cutting by the machine tool, and the discrete-event dynamics that govern the production of a series of parts queued at the machine tool.
Studies in the two fields have independently made important contributions to the operational excellence of machine tools in production environments.

The research of machining dynamics has laid a solid foundation for the operational excellence of machine tool by providing a set of stable cutting parameters, including spindle speed.
The related studies focus on the physical constraints and modeling of machining processes \cite{tyler2013analytical,tyler2016coupled1}, the stable and chatter machining parameters identification \cite{altintas2020chatter,karandikar2020stability}, tool wear and tool life prediction \cite{karandikar2015tool,karandikar2021physics}, 
cutting force modeling and measurement \cite{no2019force,gomez2019displacement}
and etc.
The machining of a single part is typically considered in these research studies.
However, few research efforts have considered the practical production requirements for hundreds or even thousands of parts from multiple accepted orders.


In the meantime, the discrete-event dynamics research for manufacturing typically considers the production scheduling of a series of parts that are in the queue of the machine tool. The operational excellence is usually achieved at the machine shop level by considering various production requirements, such as 
order scheduling on parallel machines \cite{wang2013novel,shi2015minimizing,shi2021learning}, production scheduling on batch machines \cite{shi2018customer,qin2019genetic}, 
buffers for machines \cite{zhang2017flow,ma2020workforce}, due date requirement \cite{framinan2018order} and etc.
However, these studies do not consider the physical constraints and practical operating of the machine tools, such as the machining stability and tool life, which can significantly impact the accuracy of processing times. Most existing studies assume that the processing times are fixed and known. As such, the solutions from these studies are often impractical.

Therefore, the integration of machining dynamics and discrete-event dynamics is critical and promising.
In the literature, a few studies have considered the production scheduling of machine tool,
such as machine tool feed rate scheduling for processing parts with various geometry \cite{ferry2008virtual,beudaert20115},
computer numerical control (CNC) machine scheduling with controllable processing times constrained by tool life \cite{kayan2005new,gurel2007considering},
tool replacement scheduling to balance cost of machining conditions, product quality and tool life \cite{vagnorius2010determining,zaretalab2018mathematical}.
However, to our knowledge, no research to date addresses how to integrate machining dynamics, including the machining stability, various cutting speeds and tool life, into discrete-event dynamics for the operational excellence of machine tool. 
We seek to fill this gap.

In this paper, we address the problem of integration of discrete-event dynamics and machining dynamics, denoted by IDM, to increase the profit of machine shop. The machine shop wants to reduce the total cost for producing all fulfilled orders by selecting an appropriate cutting speed for each order. 
We analyze the problem under two scenarios of the practical production constraints with increasing complexities: 1) The integrated optimization of cost and makespan; 2) The integrated optimization of cost and due date.
While only optimizing cost is the foundation, the above two scenarios increase the fidelity of our studied problem to practical production.
We show that we can take advantage of high speed machining and proper production schedule to reduce the overall cost.

We organize the rest of the paper as follows. In Section \ref{sec:problem_settings}, we present the detailed problem settings for both the discrete-event dynamics and the machining dynamics. A learning-based cost function for machine tool is given in Section \ref{sec:main_idea}. 
Different production scenarios are studied in Sections \ref{sec:IDM_C} and \ref{sec:IDM_CD}. The mathematical optimization models, analysis and algorithms are provided. 
In Section \ref{sec:experiments}, we present numerical results for the proposed methods. Section \ref{sec:conclusion} concludes this paper.

\section{Problem Settings}\label{sec:problem_settings}

In this section, we describe our problem settings of IDM in terms of orders and parts, stable spindle speed, processing time and tool life, and machine shop production and the cost.

\subsection{Discrete-Event Dynamics: Order and Part Setting}
We specify the production for $n$ accumulated orders, each consisting of a number of parts. For example, a machine shop may receive more than 10 orders from the master production plan for daily operations. We address how to schedule the production of the machine tool by identifying the best speeds for each of these orders and scheduling their sequences in the queue of the machine tool to save total production cost.

Mathematically, we define the set of $n$ orders as $\mathcal{J}=\{1,2,\dots,n\}$. Each order consists of a specific quantity of parts, denoted by $l_j$, $j \in \mathcal{J}$.
In this paper, we assume all the parts of an order are the same type, but parts for different orders can be different.
The different part types can be resulted by different situations, such as material type, part geometry, and etc.
For the single type of part in an order $j$, let the volume to be removed by machining for each part be $v_j$, such that the total volume to be removed (i.e. workload of order $j$) can be calculated by $v_j l_j$.
All the orders are released at time zero. A due date $d_j$ is associated for order $j$ before which all parts of the order must be produced.

The challenges for such discrete-event dynamics research are that the processing times of orders are various due to different cutting speeds for different part types.
To obtain the operational excellence for machine tool, the critical issue is to obtain the accurate estimation of processing time for orders.

\subsection{Machining Dynamics: Spindle Speed and Tool Life Setting}\label{sec:maching_dynamics}

Machining dynamics provides the foundation for accurate processing time estimation via the machining stability and tool life estimation under a wide variety of spindle speeds.
Without loss of generality, we adopt the milling settings throughout this paper.
The stability lobe diagram (SLD) enables to select the best stable spindle speeds at increased axial depth of cut, as shown in Fig. \ref{fig:sld}.
The machining process can either be stable, that are the green circles below the stability boundary (blue curve), and chatter (red cross). In practice, for a given tool-material combination, the SLD can be obtained by conducting impact test. 
For more details, see the SLD and its governing relationship in \cite{schmitz2018machining}.
In this paper, we assume the stable spindle speeds are selected from SLD generated by impact test.

\begin{figure}
    \centering
    \includegraphics[width=3.3in]{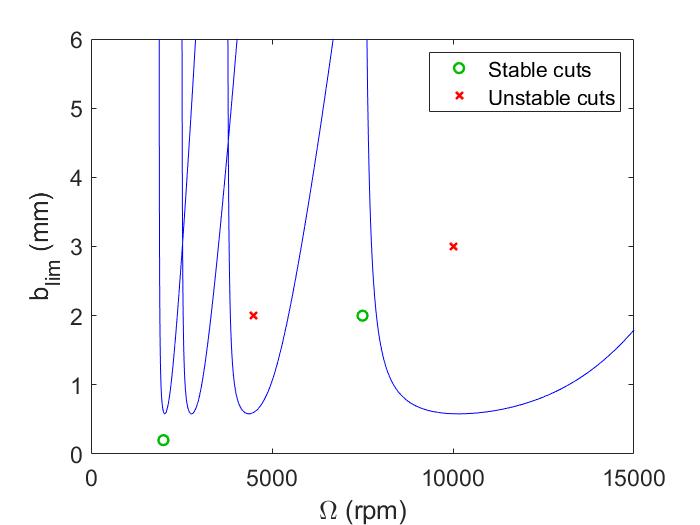}
    \caption{Stability lobe diagram (SLD) to determine the stable and unstable cutting (chatter) by choosing parameters spindle speed ($\Omega$) and axial depth of cut ($b$).}
    \label{fig:sld}
\end{figure}

Given a stable spindle speed $\Omega$ by SLD and a part with the volume to be removed as $v$, the machining time then can be calculated as
\begin{equation}
      t = \frac{v}{MRR} = \frac{v}{ a b f_t N_t \Omega},
\end{equation}
where MRR is the material removal rate defined as the average volume of material removed per unit time with a chosen stable spindle speed $\Omega$ from the SLD, $a$ and $b$ the radial and axial depth of cut, and $f_t$ is the feed per tooth. 
In this paper, we calculate $t$ by considering various $\Omega$ with fixed $a$, $b$ and $f_t$ as the same setting in \cite{karandikar2021physics}.

Notably, the tool life is decided directly given by the machining parameters, like the spindle speed $\Omega$. 
In general, this is given by the Taylor's tool life equation \cite{taylor1906art}:  
\begin{equation}\label{eq:taylor_life_model}
    V T^N = C,
\end{equation}
where the cutting speed $V$ is the peripheral velocity by $V=\pi D \Omega$, $D$ is the diameter of the tool, and $T$ is the tool life. $N$ and $C$ are constants. 
Whilst higher spindle speed can increase the machining efficiency of jobs with less machining time $t$, it leads to reduction of tool life.
Thus, there exists a balance of tool life and the machining time which will eventually impact the cost.
To obtain the tool life model is usually by machine learning methods. In this paper, we use a state-of-the-art machine learning model in \cite{karandikar2021physics} to generate the tool life as input for our method, as discussed in Section \ref{sec:main_idea}.

\subsection{Production and Cost Settings}

For production setting, we specify there is a single machine tool. There are large quantity of tool edges of the same type that are capable of processing all the $n$ orders.
In practice, it is often infeasible to complete a fractional number of a part for a tool edge as it requires to change the tool edge in the middle of the cut.
Typically, when the rest tool life is not long enough for machining a single part, the tool edge needs to be changed with a new one.
The orders are produced sequentially in the queue of the machine tool. No preemption and interruption of orders are allowed. 
Furthermore, the worker would use the same spindle speed to machine all the parts in one order, regardless of changing tool edges in between the machining process.
This is realistic and easy to implement by the on-site workers, because the quality of all parts within an order need to be maintained uniformly by keeping the same cutting conditions.

The objective is to minimize the overall cost introduced by machine tool operating cost for machining parts plus the tool edges.
In \cite{karandikar2021physics}, a cost model per part was provided as
\begin{equation}\label{eq:cost_one_part}
    C_p = r_m t + \frac{(r_m t_{ch} + C_{te}) t }{T},
\end{equation}
where we denote $r_m$, $t_{ch}$, $C_{te}$ and $C_p$ as the machine tool operating cost in \$ per unit time, the time required for changing a tool edge, the cost in \$ per tool edge and the total cost per part, respectively.
This model, however, considers the cost for a single part, rather than the overall cost of orders for hundreds or even thousands of parts made by different materials. 

In this study, we extend the above model to represent the overall cost for producing multiple parts in multiple orders as
\begin{equation}\label{eq:cost_orders}
    C_o = \sum_{j \in \mathcal{J}} \big[ r_m p_j + ( r_m t_{ch} + C_{te})q_j \big],
\end{equation}
where the first term $r_m p_j$ calculates the machine tool operating cost for processing order $j$ and $p_j$ refers to the total machining time with a chosen speed. The second term $( r_m t_{ch} + C_{te})q_j$ calculates the cost induced by changing tool edges, and $q_j$ is the number of tool edges needed for processing order $j$.

Eq. \eqref{eq:cost_orders} is the generalized form of Eq. \eqref{eq:cost_one_part} to include all orders in the practical machine shop production. In next section, we will present a new cost function through the calculation of $p_j$ and $q_j$ under the context of multiple orders in machine shop.

\section{Learning-based Cost Function for Machine Tool}\label{sec:main_idea}

The determination of cost of the orders depends on the machining time and number of tool edges used, which is directly decided by the prediction of tool life. In this section, we present a new cost function based on the tool life prediction in the sense of expectation with a state-of-the-art machine learning model.

\subsection{A Machine Learning Model for Tool Life Prediction}

Karandikar et al. (2021) \cite{karandikar2021physics} developed a physics-guided logistic regression model to predict tool wear probability with cutting speed $V$ and cutting time $T$ as inputs. Through introducing into machine learning Taylor's power law-based tool life description (see Eq. \eqref{eq:taylor_life_model}), the model has a powerful performance in prediction accuracy. The logistic decision boundary of the model, that is, when the probability of tool worn is 0.5, provides an accurate estimation for tool life in the sense of expectation as
\begin{equation}
    \log(T) = -\frac{\theta_1}{\theta_2} \log(V) -\frac{\theta_0}{\theta_2},
\end{equation}
where the coefficients $\theta_0=-89.57$, $\theta_1=13.57$ and $\theta_2=5.26$ are estimated through experimental data for a single-insert endmill (Kennametal KICR073SD30333C) with $D=19.05$ mm diameter and a square uncoated carbide insert (Kennametal 107,888,126 C9 JC). 
Equivalently, this model can be transformed into the below Taylor's tool life model after using logarithmic transformation, 
\begin{equation}\label{eq:tool_life}
      T = \bigg( \frac{C}{V} \bigg)^{1/N} = \bigg(\frac{C}{\pi D \Omega} \bigg)^{1/N},
\end{equation}
where $N=0.388$ and $C=735.2$ m/min using the above values of $\theta_0$, $\theta_1$ and $\theta_2$. The logistic model in the logarithmic and original space of $V$ and $T$ is shown in the two panels of Fig. \ref{fig:tool_life}. The gray scale indicates the probability of tool worn with a cutting speed $V$ (horizontal axis) and cutting time $T$ (vertical axis). We will use Eq. \eqref{eq:tool_life} throughout this paper for tool life prediction.

\begin{figure}[t]
    \centering
    \includegraphics[width=5in]{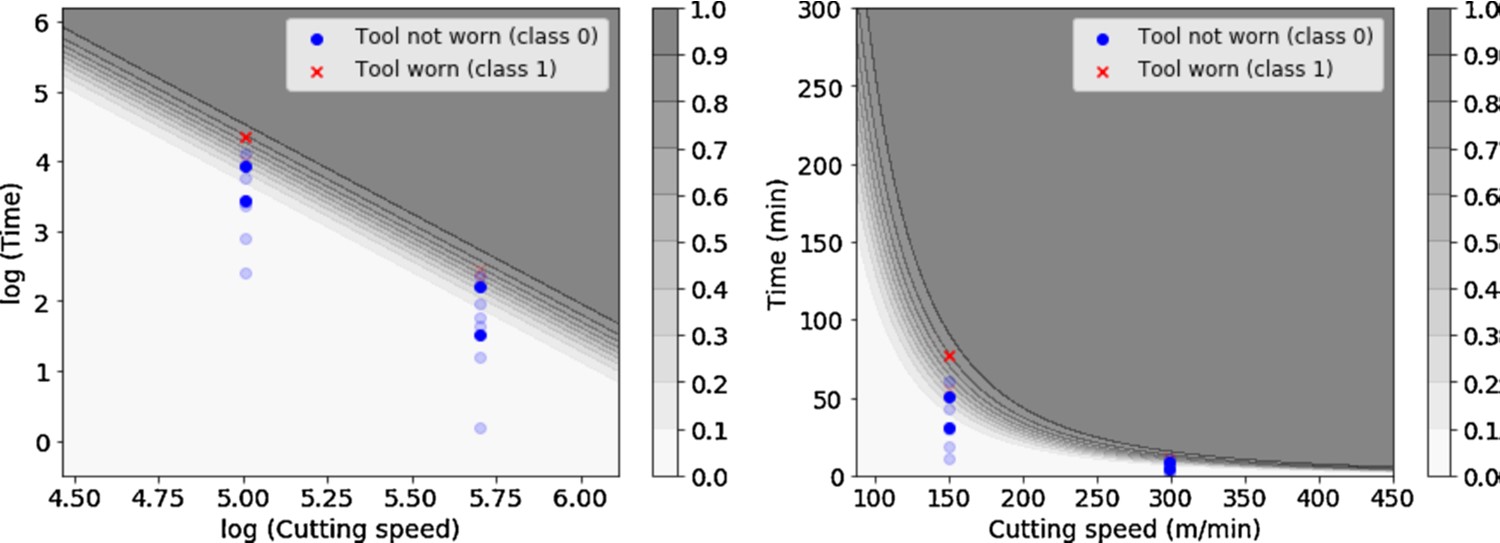}
    \caption{Probability of tool worn as a function of cutting speed from the trained logistic model; the train data points are shown in light and the test data points are shown in dark. The left panel shows the results in the logarithmic space and the right panel in the original space. Figures are reproduced from \cite{karandikar2021physics}.}
    \label{fig:tool_life}
\end{figure}

\subsection{Calculation of Learning-based Cost Function}\label{sec:data_pre}

Based on the learning-based tool life model in \cite{karandikar2021physics}, a cost function is developed for multiple orders. As the parts for different orders can be made by different materials, the SLD of the same tool edge for different orders (parts) are different. 
As a result, the candidate stable spindle speeds for orders are also different.
Thus, we define the matrix of spindle speed as
\begin{equation}
    {\bf \Omega} =
    \left[
    \begin{array}{ccccc}
         \Omega_{11} & \Omega_{12}  & \cdots & \Omega_{1n}\\
         \Omega_{21} & \Omega_{22}  & \cdots & \Omega_{2n}\\
         \vdots      &  \vdots      & \ddots & \vdots  \\
         \Omega_{m1} & \Omega_{m2}  & \cdots & \Omega_{mn}\\
    \end{array}
    \right],
\end{equation}
where the entry $\Omega_{ij}$ refers to the spindle speed for processing order $j$ with spindle speed $i$, $i \in \mathcal{I}$, and $\mathcal{I}=\{1,2,\dots,m\}$ is defined as the index set of available stable spindle speed.

For a given spindle speed $\Omega_{ij}$ to process order $j$, the tool life $T_{ij}$ is calculated by Eq. \eqref{eq:tool_life}. Also, the total machining time $p_{ij}$ for processing order $j$ with spindle speed $\Omega_{ij}$ calculated as 
\begin{equation}\label{eq:order_process_time}
    p_{ij} = \frac{v_j l_j} {a b f_t N_t \Omega_{ij}},
\end{equation}
where $v_j l_j$ indicates total volume of order $j$ by multiplying its part quantity $l_j$ with the part volume $v_j$.
When arranged $T_{ij}$ and $p_{ij}$ in matrix form, we can obtain the below matrices for tool life and machining time:
\begin{equation}
     {\bf T} =
    \left[
    \begin{array}{ccccc}
         T_{11} & T_{12}  & \cdots & T_{1n}\\
         T_{21} & T_{22}  & \cdots & T_{2n}\\
         \vdots      &  \vdots      & \ddots & \vdots  \\
         T_{m1} & T_{m2}  & \cdots & T_{mn}\\
    \end{array}
    \right]
    \text{  and  }
     {\bf P} =
    \left[
    \begin{array}{ccccc}
         p_{11} & p_{12}  & \cdots & p_{1n}\\
         p_{21} & p_{22}  & \cdots & p_{2n}\\
         \vdots      &  \vdots      & \ddots & \vdots  \\
         p_{m1} & p_{m2}  & \cdots & p_{mn}\\
    \end{array}
    \right],
\end{equation}

Using $\bf T$, we can now compute the number of tool edges needed for each order with different spindle speeds as matrix 
\begin{equation}
    {\bf Q} =
    \left[
    \begin{array}{ccccc}
         q_{11} & q_{12}  & \cdots & q_{1n}\\
         q_{21} & q_{22}  & \cdots & q_{2n}\\
         \vdots      &  \vdots      & \ddots & \vdots  \\
         q_{m1} & q_{m2}  & \cdots & q_{mn}\\
    \end{array}
    \right],
\end{equation}
where each entry is calculated by
\begin{equation}\label{eq:tool_qty}
    q_{ij} = \Biggl\lceil \frac{l_j}{ \big \lfloor T_{ij} / t_{ij} \big \rfloor } \Biggl\rceil, \quad\text{and}\quad t_{ij} = \frac{v_j } {a b f_t N_t \Omega_{ij}}. 
\end{equation}
The term $t_{ij}$ is the machining time for a single part of order $j$ with spindle speed $\Omega_{ij}$. Thus, with floor function, $\lfloor T_{ij} / t_{ij} \big \rfloor$ refers to the number of parts of order $j$ that can be finished by a single tool edge. Finally, the number of tool edges for processing $l_j$ parts can be obtained by the ceiling function.

We can further obtain the summation of total machining time and tool edge change time for each order under all its candidate spindle speeds. This summation, denoted by $h_{ij}$, represents the total processing time occupied by order $j$ if with speed $i$:
\begin{equation}\label{eq:sum_time}
    h_{ij} = p_{ij} + t_{ch} q_{ij}.
\end{equation}
The associated total time matrix is denoted by
\begin{equation}
    {\mathbf{H}} =
    \left[
    \begin{array}{ccccc}
        h_{11} & h_{12}  & \cdots & h_{1n}\\
        h_{21} & h_{22}  & \cdots & h_{2n}\\
        \vdots      &  \vdots      & \ddots & \vdots  \\
        h_{m1} & h_{m2}  & \cdots & h_{mn}\\
    \end{array}
    \right].
\end{equation}

Finally, with matrices $\bf P$, $\bf Q$ and $\bf H$, we can obtain the cost function for processing all orders as 
\begin{equation}\label{eq:cost_function}
    f_{cost}(x)=\sum_{j\in \mathcal{J}} \sum_{i \in \mathcal{I}} (r_m h_{ij} + C_{te}q_{ij}) x_{ij}.
\end{equation}
Note here we introduce an indicator $x_{ij}$ defined as
\begin{equation}\label{eq:x_var}
    x_{ij} := \left\{
   \begin{aligned}
   & 1, \quad\text{ if speed $i$ is chosen for processing order $j$},\\
   & 0, \quad\text{ otherwise.}
   \end{aligned}
   \right.
\end{equation}
Eq. \eqref{eq:cost_function} integrates both the machining dynamics, that is, term $r_m h_{ij} + C_{te}q_{ij}$ based on the cost function for parts of orders, and the discrete-event dynamics, that is, discrete variable $x_{ij}$ in Eq. \eqref{eq:x_var} for choosing different speeds for different orders to achieve the reduction of production cost. 

In next section, we will optimize the speed selection through controlling $x_{ij}$ under the discrete-event dynamics environment of machine shop such that $f_{cost}(x)$ can be minimized for processing all orders. 
Before that, we show an illustrative example.

\subsection{Illustration of Cost Saving by the Proposed Integration}\label{sec:motivation}

As an illustrative example of our proposed integration, we show two scenarios in Fig. \ref{fig:motivation} for processing an order: with faster speed in (a) obtained by the expected cost method in \cite{karandikar2021physics}, and slower speed in (b) by solving our developed \ref{M1} model in Section \ref{sec:IDM_C}.
The light grey period is for tool machining, the light green period is for changing the tool edge. The red dashed sections refer to the waste time that is not enough for machining a single part before its tool life comes to end, while the one in the last tool edge is wasted because all parts are done.

\begin{figure}[h]
    \centering
    \includegraphics[width=3.3in]{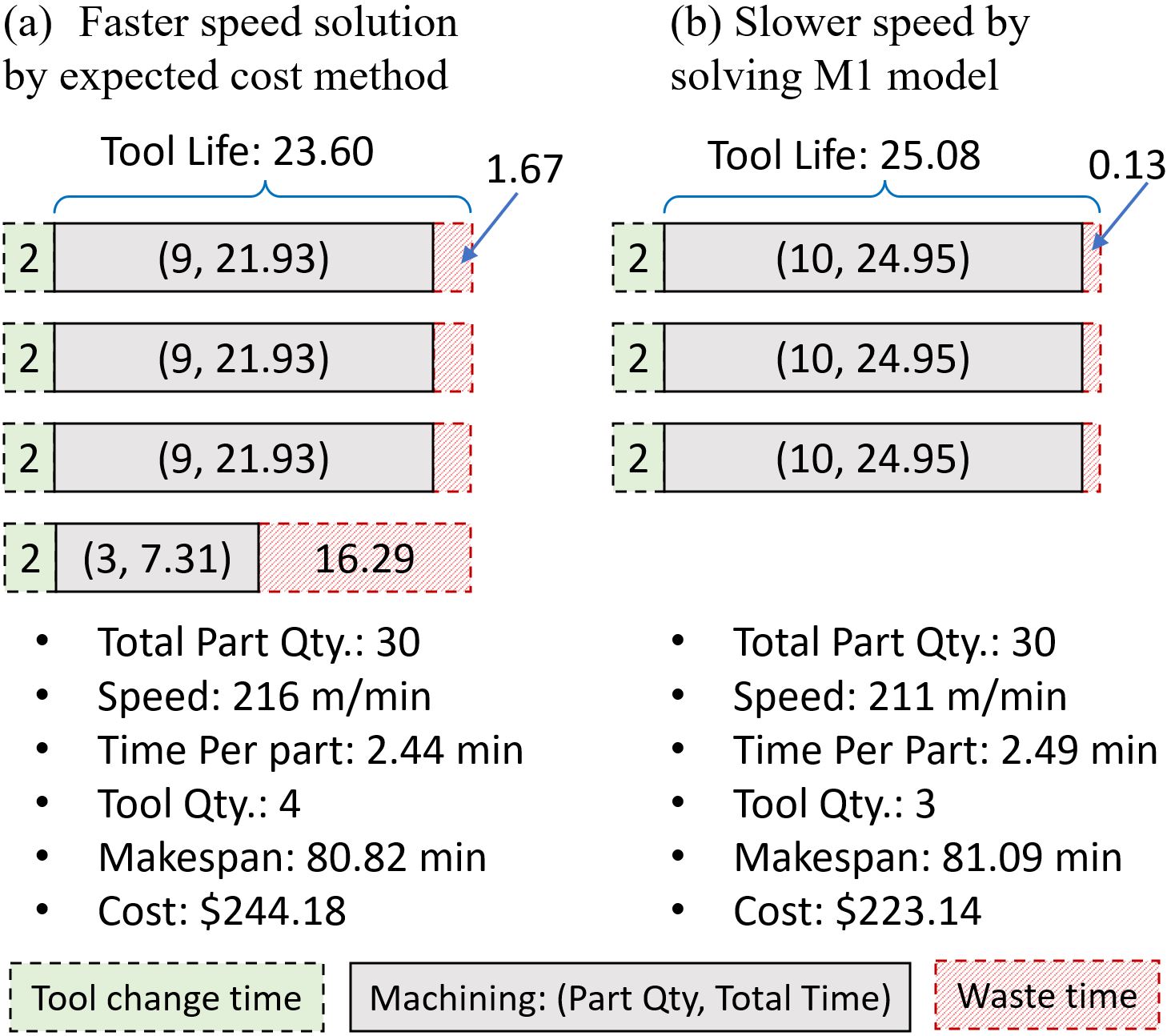}
    \caption{Illustrative example for choosing different speeds to process multiple parts of an order by the expected cost method in \cite{karandikar2021physics} and solving our developed model \ref{M1} in Section \ref{sec:IDM_C}. 
    By choosing slower speed, we process one more part per tool edge in (b) in comparison with (a). The last three parts on the 4th tool edge of (a) are thus moved back to the 3rd tool edge. This leads to saving the last tool edge and reducing the waste time for each tool edge in (a), and importantly, reducing the cost from \$244.18 to \$223.14. See details in Section \ref{sec:experiments}.}
    \label{fig:motivation}
\end{figure}

By choosing slower speed in (b), we process one more part per tool edge in (b) in comparison with (a), namely from 9 part per tool edge to 10. The last three parts on the 4th tool edge of (a) are thus moved back to be processed with the 3rd tool edge, saving the last tool edge and reducing the waste time for each tool edge in (a).
Notably, the production cost is reduced from \$244.18 to \$223.14. The makespan is increased from 80.82 min to 81.09 min as a sacrifice of the reduced cost. Such trade-off between cost and makespan is meaningful in practical production. 
Note Figure \ref{fig:motivation} only gives the case when smaller speed will lead to cost reduction, it is more often and realistic to reduce cost by increasing the speed for high speed machining, and this exactly shows our motivation for this study.

\section{Integrated Optimization of Cost and Makespan}\label{sec:IDM_C}

We denote our problem when only cost is considered as IDM-C, which is studied in Sections \ref{sec:IDM_C_model} - \ref{sec:IDM_C_algo} . In section \ref{sec:IDM_CM},
we study the IDM problem when integrated optimization of cost and makespan is considered, denoted by IDM-CM.
Integer optimization (IO) models are presented for both problems.
For clarity, we first summarize all the notations that will be used for all the models proposed in this paper throughout.

\vspace{0.3cm}
\noindent \textbf{Sets:}
\begin{description}
    \item[\quad $\mathcal{J}:$] set of orders, $\mathcal{J}=\{1,2,\dots,n\}$. Each order $j$ consists of a quantity of $l_j$ parts that are the same part type with the volume to be removed as $v_j$;
    \item[\quad $\mathcal{I}:$] index set of stable candidate spindle speed for arbitrary order, $\mathcal{I}=\{1,2,\dots, m\}$;

    \item[\quad $\bf \Omega:$] ${\bf \Omega} \in \mathbb{R}^{m\times n}$ is the matrix of spindle speed where entry $\Omega_{ij}$ refers to the spindle speed $i$ for processing order $j$, $\forall j \in \mathcal{J}$, $i \in \mathcal{I}$. A column $\Omega_{\boldsymbol{\cdot} j}$ refers to all candidate spindle speeds that are stable for processing order $j$;
    
    \item[\quad $\mathbf{T}:$] $\mathbf{T} \in \mathbb{R}^{m \times n}$ is the matrix of tool edge life if selecting spindle speed $\Omega_{ij}$ for order $j$, $\forall j \in \mathcal{J}$, $i \in \mathcal{I}$. Each entry $T_{ij}$ is calculated by Eq. \eqref{eq:tool_life};
    
    \item[\quad $\mathbf{P}:$] $\mathbf{P}\in \mathbb{R}^{m \times n}$ is the matrix of total machining time for all orders where entry $p_{ij}$ refers to the total machining time if selecting $\Omega_{ij}$ to process order $j$, as in Eq. \eqref{eq:order_process_time};    
    
    \item[\quad $\mathbf{Q}:$] $\mathbf{Q} \in \mathbb{Z}_+^{m \times n}$ is the matrix of the number of tool edges needed for order $j$ with entry $q_{ij}$ calculated by Eq. \eqref{eq:tool_qty};

    \item[\quad $\mathbf{H}:$] $\mathbf{H} \in \mathbb{R}^{m \times n}$ is the matrix of total processing time where entry $h_{ij}$ consists of two parts: total machining time $p_{ij}$ and the tool change time $t_{ch} q_{ij}$, as calculated by Eq. \eqref{eq:sum_time}.
\end{description}

\noindent \textbf{Parameters:}
\begin{description}
    \item[\quad $r_m:$] machine tool operating cost ratio in \$ per unit time;
    \item[\quad $t_{ch}:$] time required for changing a tool edge;
    \item[\quad $C_{te}:$] cost in \$ per tool edge.
\end{description}

\noindent \textbf{Decision variables:}
\begin{description}
    \item[\quad $x_{ij}:$] $x_{ij}=1$ if spindle speed $\Omega_{ij}$ is selected for processing order $j$, $\forall j \in \mathcal{J}$ and $i \in \mathcal{I}$. Otherwise, $x_{ij}=0$.
\end{description}

\subsection{Model with Only Cost Consideration}\label{sec:IDM_C_model}

When only cost is under consideration, the below IO model \ref{M1} aims to find the optimal spindle speed for each order, such that the cost function can be minimized.
\begin{align}
\label{M1} \tag{M1} \quad \min \quad  
& f_{cost}(x) \\
    \label{cons:M1_1}	\textrm{s.t.} \quad & \sum_{i \in \mathcal{I}} x_{ij} = 1, ~~~~~~~~~~ \forall j \in \mathcal{J},\\
	\label{cons:M1_2} & x_{ij} \in \{0,1\}, ~~~~~~~~~~~ \forall i \in \mathcal{I}, j \in \mathcal{J}.
\end{align}
The objective is to minimize the overall cost $f_{cost}(x)$ of producing $n$ orders. Constraints \eqref{cons:M1_1} ensure that each order must be processed by only one spindle speed. Constraints \eqref{cons:M1_2} provide the binary variable restriction on $x_{ij}$.

We show the constraint matrix of formulation~\ref{M1} is totally unimodular. Thus, the IO model can be solved as a linear programming (LP) model.
To do this, we arrange the variables in the sequence of $x_{11}, x_{21},\dots, x_{m1};x_{12},\dots x_{m2}; \dots; x_{1n},\dots, x_{mn}$ as columns of the constraint matrix. The constraint matrix $\mathbf{A}\in \mathbb{Z}^{n\times mn}$ for constraints (\ref{cons:M1_1}) is as below:
\begin{equation}
\setlength{\arraycolsep}{1pt}
\bf A=
\begin{bNiceMatrix}[first-row]
1^{th} & \dots  & m^{th} & (m+1)^{th}  & \dots & 2m^{th}   & \dots &  ((n-1)m+1)^{th}  & \dots  &   nm^{th}  \\
1 & \dots  & 1 & 0  & \dots & 0   & \dots & 0  & \dots  &   0  \\
        0 & \dots  & 0 & 1  & \dots & 1   & \dots & 0  & \dots  &   0  \\
        \vdots & \vdots  & \vdots & \vdots  & \vdots & \vdots   & \vdots & \vdots  & \vdots  &   \vdots  \\
        0 & \dots  & 0 & 0  & \dots & 0   & \dots & 1  & \dots  &   1  \\
\end{bNiceMatrix}.
\end{equation}

The above 0, 1 matrix $\bf A$ is totally unimodular because it satisfies the following two conditions: (1) there are at most two nonzero elements in each column of $\bf A$, which is number 1 and there is only one entry for each column to be 1; (2) $\bf A$ admits an equitable row-bicoloring. In particular, the rows of $\bf A$ can be partitioned into two sets, say ``red set" as row 1 and ``blue set" as rows 2 to n. It's easy to check that the sum of the red rows minus the sum of the blue rows is a vector with entries $\pm1$.

As a result, the polyhedron for formulation \ref{M1} is integral, which has integral extreme points. Thus, it always admits an integral solution and is equivalent to solving the LP relaxation: 
\begin{align}
\label{M1-LP} \tag{M1-LP} \quad \min \quad  & f_{cost}(x) \\
    \label{cons:LP_NC1}	\textrm{s.t.} \quad & \sum_{i \in \mathcal{I}} x_{ij} = 1, ~~~~~~~~~~ \forall j \in \mathcal{J},\\
    \label{cons:LP_NC2}	& x_{ij} \in [0,1], ~~~~~~~~~~~ \forall i \in \mathcal{I}, j \in \mathcal{J},
\end{align}
where we relax the binary variable constraints $x_{ij} \in \{0,1\}$ into the interval $x_{ij}\in[0,1]$. Next, we develop an exact algorithm.

\subsection{Decomposition-based Greedy Algorithm}\label{sec:IDM_C_algo}

Formulation \ref{M1} naturally adopts a decomposition for both the objective and the constraints for each order. 
While every order must be assigned with a spindle speed, there are no additional constraints that couples the assignment of spindle speed for different orders.
Therefore, the assignment of spindle speed for each order is independent.
That is, for each order $j$, we can select the optimal speed from $m$ candidates that minimize its cost $r_m h_{ij} + C_{te}q_{ij}$ in the objective.
In general, we denote the described procedure as Decomposition-based Greedy (DG) and summarize as the below Algorithm \ref{alg:DG}.

\begin{algorithm}[h!]
\caption{The DG algorithm}\label{alg:DG}
\setstretch{1.2}
\SetKwInOut{Input}{Input}
\SetKwInOut{Output}{Output}
\Input{Data matrices $\bf \Omega$; order set $\mathcal{J}$ and associated $l_j$ and $v_j$; the cost parameters $r_m$, $t_{ch}$ and $C_{te}$.}
Initialize $f_{cost}(x) = 0$; $x_{ij} = 0, \forall i \in \mathcal{I}, j\in \mathcal{J}$\;
Calculate matrices $\mathbf{T}$, $\mathbf{Q}$, $\mathbf{P}$ and $\bf H$ with Eqs. \eqref{eq:tool_life} - \eqref{eq:sum_time}\;
\For(\tcc*[f]{Decompose as $n$ subproblems}){$j\in \mathcal{J}$}
{
    \For(\tcc*[f]{Find best speed for $j$}){$i \in \mathcal{I}$}
   {
        Compute $Obj_{ij} = r_m h_{ij} + C_{te}q_{ij}$\;
   }
    Find the best speed $i^*$ of job $j$ by $i^* := \argmin_{i \in \mathcal{I}} Obj_{ij}$\;
   Update $x_{i^*j}=1$\;
   $f_{cost}(x) \leftarrow f_{cost}(x) + Obj_{i^* j}$
}
\textbf{Output} Optimal cost $f_{cost}(x)$, speed $x_{ij}, \forall i \in \mathcal{I}$, $j\in \mathcal{J}$
\end{algorithm}

For Algorithm \ref{alg:DG}, we have the below lemma to show the optimality of its solution.

\begin{lemma}
Given $n$ orders and $m$ candidate speeds, Algorithm~\ref{alg:DG} solves problem IDM-C optimally in $O(18mn+n)$ time. 
\end{lemma}

\begin{proof}
To prove DG algorithm is optimal, we do it by contradiction. Suppose the solution of DG is not optimal, then there must exist at least one order, denoted by $\hat{j}$, that selects a spindle speed with smaller cost $r_m h_{ij} + C_{te}q_{ij}$ than $\hat{j}$. This contradicts with steps 7-9 of DG, which selects the spindle speed with the smallest cost. Thus, Algorithm \ref{alg:DG} is an exact algorithm.

Next, we consider the solution time of Algorithm \ref{alg:DG}. 
For preparing $\mathbf{T}$, $\mathbf{Q}$, $\mathbf{P}$ and $\bf H$,
it takes 2 operations to calculate $T_{ij}$, 6 operations to calculate $q_{ij}$, 3 operations to calculate $p_{ij}$ and 2 operations to calculate $h_{ij}$. Together with the coupled loops, this takes $14mn$ times.
For each $j$, it takes $3m$ operations to compute $Obj_{ij}$ for all $i\in \mathcal{I}$, $m$ operations to find the speed $i^*$ with minimum $Obj_{ij}$, and 1 operation to update the overall cost $f_{cost}(x)$, respectively. To sum up, we can get the total time complexity of Algorithm \ref{alg:DG} is $O(18mn+n)$, which completes the proof.
\end{proof}

\subsection{Integration with Makespan}\label{sec:IDM_CM}

Makespan constraint refers to the hard constraint that all the orders must be completed within a given time period. For example, this can describe the fixed time length for the worker shift. To model problem IDM-CM, we need an additional parameter:
\begin{description}
        \item[\quad $c:$] a given constant of the upper bound for the makespan.
\end{description}

With this additional parameter, we can model problem IDM-CM by modifying formulation \ref{M1} with makespan constraint as below IO model, denoted by \ref{M2}:
\begin{align}
\label{M2} \tag{M2} \quad \min \quad  & f_{cost}(x) \\
    \label{cons:M2-0}\textrm{s.t.} \quad & \sum_{i \in \mathcal{I}} x_{ij} = 1, ~~~~~~~~~~~ \forall j \in \mathcal{J},\\
    \label{cons:M2-1} & \sum_{j \in \mathcal{J}} \sum_{i \in \mathcal{I}} h_{ij} x_{ij} \leq c,\\
	 \label{cons:M2-2} & x_{ij} \in \{0,1\}, ~~~~~~~~~~~ \forall i \in \mathcal{I}, j \in \mathcal{J}.
\end{align}
Constraint \eqref{cons:M2-1} ensures the makespan of producing all the orders not exceeding the given time length $c$.
The makespan constraint adds additional complexity for the problem. We remark the complexity of problem IDM-CM is still open.

\section{Integrated Optimization of Cost and Due Date}\label{sec:IDM_CD}

We consider problem IDM when the integrated optimization of cost and due date is to be achieved, denoted by IDM-CD. In this case, the sequence of orders needs to be considered. 
We provide three categories of model based on the choices of decision variables. That are, completion time variables \eqref{sec:model_D1}, linear order variables \eqref{sec:model-D2} and positional variables \eqref{sec:model-D3} of orders, respectively.
In the first category, we propose a mixed-integer optimization (MIO) model. For each of the latter two, we present a quadratic integer optimization model and an MIO model by applying linearization technique.

\subsection{Modeling with Completion Time Variables}\label{sec:model_D1}

In the first model, we use completion time variables to indicate the order sequence for problem IDM-CD. The needed additional variables are defined as follows:
\begin{description}
        \item[\quad $y_{jk}:$] $y_{jk}=1$ if order $j$ is processed before order $k$, $\forall j, k \in \mathcal{J}$ and $j \neq k$. Otherwise $y_{jk}=0$;
        \item[\quad $C_j:$] completion time of order $j$, $\forall j \in \mathcal{J}$.
\end{description}

The problem IDM-CD can be formulated as below.
\begin{align}
\label{M3-1} \tag{M3-1} \quad \min \quad  
& f_{cost}(x)\\
    \label{cons:M3-1-1}	\textrm{s.t.} \quad & \sum_{i \in \mathcal{I}} x_{ij} = 1, ~~~~~~~~~~ \forall j \in \mathcal{J},\\
    \label{cons:M3-1-2} & C_j \geq \sum_{i \in \mathcal{I}} h_{ij} x_{ij}, ~~~~~~~~~~~ \forall j \in \mathcal{J},\\
    \label{cons:M3-1-3} & C_j + \sum_{i \in \mathcal{I}} h_{ik} x_{ik} \leq C_k + M(1-y_{jk}),   ~ \forall j<k \in \mathcal{J}, \\
    \label{cons:M3-1-4} & C_k + \sum_{i \in \mathcal{I}} h_{ij} x_{ij} \leq C_j + My_{jk}, ~~~~~~~ \forall j<k \in \mathcal{J}, \\
    \label{cons:M3-1-5} &  C_j \in[0,d_j], ~~~~~~~~~~~ \forall j \in \mathcal{J},\\
	\label{cons:M3-1-6} & x_{ij}, y_{jk} \in \{0,1\}, ~~~~~~~~~~~ \forall i \in \mathcal{I}, j,k \in \mathcal{J}.
\end{align}
Constraints \eqref{cons:M3-1-2} ensure that the completion time of each order is no later than its processing time. Constraints \eqref{cons:M3-1-3} and \eqref{cons:M3-1-4} are disjunctive constraints which ensure that either order $j$ is processed before order $k$, i.e. $y_{jk}=1$ in constraints \eqref{cons:M3-1-3}, or order $k$ is before order $j$, that is $y_{jk}=0$ in constraints \eqref{cons:M3-1-4}. $M$ is a sufficiently large number. Constraints \eqref{cons:M3-1-5} enforce the completion time is non-negative and no larger than the due date. In this formulation, the $M$ is a natural upper bound for $C_j$ for arbitrary $j \in \mathcal{J}$. Thus, we can set $M=\sum_{j\in \mathcal{J}} \max_{i\in \mathcal{I}}\{h_{ij}\}$.
The total number of variables for formulation \ref{M3-1} is $n^2+mn+n$, which includes $n^2 + mn$ binary variables and $n^2+3n$ constraints.

\begin{remark}
If $d_j=c$ for $j\in \mathcal{J}$, problem IDM-CD is reduced to problem IDM-CM.
\end{remark}

To check the above remark is valid, it's obvious to see that when $d_j=c$ for $j\in \mathcal{J}$, then the makespan of any feasible solutions to problem IDM-CD also satisfy the makespan limit $c$. Thus, problem IDM-CM is a special case of problem IDM-CD. It is more general for us to study problem IDM-CD.

\subsection{Modeling with Linear Ordering Variables}\label{sec:model-D2}

\subsubsection{Proposed Integer Optimization Model}

This formulation is based on the linear order variables $y_{jk}$ and $y_{kj}$ that either of them should be equal to 1 for arbitrary $j<k\in \mathcal{J}$.
The following formulation \ref{M3-2} is considered.
\begin{align}
\label{M3-2} \tag{M3-2} \quad \min \quad  
    & f_{cost}(x) \\
    \label{cons:M3-2-1}	\textrm{s.t.} \quad & \sum_{i \in \mathcal{I}} x_{ij} = 1, ~~~~~~~~~~ \forall j \in \mathcal{J},\\
    \label{cons:M3-2-2}& y_{jk} + y_{kj} = 1,   ~~~~~~~~~~ \forall j, k \in \mathcal{J}, j<k,\\
    \label{cons:M3-2-3}& y_{jk} + y_{kl} + y_{lj} \leq 2, ~~~~~~~~~~ \forall j, k, l \in \mathcal{J}, j \neq k \neq l, \\
    \label{cons:M3-2-4}& \sum_{j \in \mathcal{J}} \sum_{i \in \mathcal{I}} h_{ij} x_{ij} y_{jk} + \sum_{i \in \mathcal{I}} h_{ik} x_{ik} \leq d_k,   ~~~~~ \forall k \in \mathcal{J},\\
	\label{cons:M3-2-5} & x_{ij}, y_{jk} \in \{0,1\}, ~~~~~~~~~~~ \forall i \in \mathcal{I}, j,k \in \mathcal{J}.
\end{align}
Constraints \eqref{cons:M3-2-2} are a set of conflict constraints, which ensure that either order $j$ is processed before order $k$ or vice versa. Note here $j < k$ because of symmetry. Constraints \eqref{cons:M3-2-3} indicate the transitivity constraints that ensure a linear order of three arbitrary orders. 
Constraints \eqref{cons:M3-2-4} indicate the starting time of order $k$ plus its processing time must be no larger than its due date.
The total number of variables of formulation \ref{M3-2} with $m$ speeds and $n$ orders is $n^2+mn$, which are all binary variables and involve $n^3-2.5n^2+3.5n$ constraints.

Note that \ref{M3-2} is a quadratically constrained IO formulation due to constraints \eqref{cons:M3-2-4}. 
Considering that the current optimization solvers like \texttt{CPLEX} cannot efficiently deal with this kind of nonlinear problem, we adopt linearization strategies for this formulation in the following section.

\subsubsection{Reformulation with Linearization Technique}\label{sec:MIO-D2}

To enhance the solvability of formulation \ref{M3-2}, we first apply the linearization technique proposed in \cite{glover1974converting} to derive an MIO model, which is equivalent to \ref{M3-2}.
For arbitrary $j\in \mathcal{J}$, $k \in \mathcal{J}$, let  $\alpha_{jk}=(\sum_{i\in \mathcal{I}} h_{ij}x_{ij})y_{jk}$ and substitute the cross-product $\sum_{i\in \mathcal{I}} h_{ij} x_{ij} y_{jk}$ in \ref{M3-2}
using the following inequalities:
\begin{equation}\label{eq:linearization-1}
    \left\{
   \begin{aligned}
   &\sum_{i \in \mathcal{I}} h_{ij} x_{ij} - M(1-y_{jk}) \leq \alpha_{jk}\leq \sum_{i \in \mathcal{I}} h_{ij} x_{ij} + M(1-y_{jk}),\\
   &-M y_{jk} \leq \alpha_{jk}\leq M y_{jk},
   \end{aligned}
   \right.
\end{equation}
where $M$ is a sufficiently large number.
For arbitrary $j,k\in \mathcal{J}$, if $y_{jk}=1$, then $\alpha_{jk}=\sum_{i\in \mathcal{I}} h_{ij}x_{ij}$, and the inequalities become $\sum_{i \in \mathcal{I}} h_{ij} x_{ij} \leq \alpha_{jk}\leq \sum_{i \in \mathcal{I}} h_{ij} x_{ij}$ and $-M \leq \alpha_{jk}\leq M$. Clearly, the equivalence holds. Otherwise, if $y_{jk}=0$, the inequalities are $\sum_{i \in \mathcal{I}} h_{ij} x_{ij} - M \leq \alpha_{jk}\leq \sum_{i \in \mathcal{I}} h_{ij} x_{ij} + M$ and $0 \leq \alpha_{jk}\leq 0$, which indicates $\alpha_{jk}=0$.
Thus, we get an equivalent model to \ref{M3-2}.

We further reduce the above inequality constraints \eqref{eq:linearization-1} as
\begin{equation}\label{eq:linearization-2}
    \left\{
   \begin{aligned}
   & \sum_{i \in \mathcal{I}} h_{ij}x_{ij}-M(1-y_{jk}) \leq \alpha_{jk},\\
   & 0\leq \alpha_{jk}\leq My_{jk}.
   \end{aligned}
   \right.
\end{equation}
Thus, we obtain a new linearized formulation \ref{M3-2L} as below:
\begin{align}
\label{M3-2L} \tag{M3-2-L} \quad \min \quad  
& f_{cost}(x) \\
    \label{cons:M3-2L-1}	\textrm{s.t.} \quad 
    & \text{Constraints } \eqref{cons:M3-2-1}-\eqref{cons:M3-2-3},\\
    \label{cons:M3-2L-4}& \sum_{j \in \mathcal{J}} \alpha_{jk} + \sum_{i \in \mathcal{I}} h_{ik} x_{ik} \leq d_k,   ~~~~~~~~~~ \forall k \in \mathcal{J}, \\
	\label{cons:M3-2L-5} & \sum_{i \in \mathcal{I}} h_{ij} x_{ij} -M(1-y_{jk}) \leq \alpha_{jk} \leq My_{jk}, ~~~ \forall j, k \in \mathcal{J}, \\ 
    \label{cons:M3-2L-6}& \alpha_{jk} \geq 0, ~~~~~~~~~~~ \forall j, k \in \mathcal{J},\\  
    \label{cons:M3-2L-7}& x_{ij}, y_{jk} \in \{0,1\}, ~~~~~~~~~~~ \forall i \in \mathcal{I}, j,k \in \mathcal{J}.
\end{align}
This formulation is obtained by eliminating the right-hand-side of the first inequality and the left-hand-side of the second inequality in \eqref{eq:linearization-1}. The value of $M$ is set as the maximum total processing time of the orders, that is $M = \max_{i,j} \{h_{ij}\}$. 
The total number of variables is $2n^2 + mn$, which involves $n^2 + mn$ binary variables and $n^3 + 0.5n^2 +3.5n$ constraints. For formulations \ref{M3-2L} and \ref{M3-2}, we have the below Lemma \ref{lemma:linearization_equivalence}.

\begin{lemma}\label{lemma:linearization_equivalence}
For arbitrary optimal solution to \ref{M3-2L}, there exists an optimal solution to \ref{M3-2} with the same objective value.
\end{lemma}

\begin{proof}
For arbitrary optimal solution $(x^*, y^*, \alpha^*)$ to \ref{M3-2L}, we can show $(x^*, y^*)$ as an optimal solution to \ref{M3-2} with the same objective value.

We first show $(x^*, y^*)$ is a feasible solution to \ref{M3-2}. 
For arbitrary $j, k \in \mathcal{J}$, if $y_{jk}^* = 0$, then $\alpha_{jk}^* = 0$ based on constraints \eqref{cons:M3-2L-5} and \eqref{cons:M3-2L-6}. Then we have $\sum_{i \in \mathcal{I}} h_{ik} x_{ik}^* \leq d_k$ in constraints \eqref{cons:M3-2L-4}. This indicates $(x^*, y^*)$ satisfies constraints \eqref{cons:M3-2-4} in formulation \ref{M3-2} when $y_{jk}^*=0$. If $y_{jk}^*=1$, then by substituting into constraints \eqref{cons:M3-2L-5} first and then \eqref{cons:M3-2L-4}, we have
\begin{align}
    &\sum_{i \in \mathcal{I}} h_{ij} x_{ij}^* \leq \alpha_{jk}^* \leq M,\\
  \label{eq:lemma1}  & \sum_{j \in \mathcal{J}} \sum_{i \in \mathcal{I}} h_{ij} x_{ij}^* +  \sum_{i \in \mathcal{I}} h_{ik} x_{ik}^* \leq \sum_{j \in \mathcal{J}} \alpha_{jk}^* + \sum_{i \in \mathcal{I}} h_{ik} x_{ik}^* \leq d_k.
\end{align}
Inequalities \eqref{eq:lemma1} indicate $(x^*, y^*)$ satisfies constraints \eqref{cons:M3-2-4} when $y_{jk}^*=1$. To sum up, $(x^*, y^*)$ is a feasible solution to \ref{M3-2}.

Next, we prove the objective value of $(x^*, y^*)$ is less than or equal to that of any feasible solutions to \ref{M3-2}. Let $\mathcal{S}_2$ be the solution space of formulation \ref{M3-2}, i.e.,
\begin{equation}
    \mathcal{S}_2 := \{ (x,y) ~:~ \text{Constraints \eqref{cons:M3-2-1} - \eqref{cons:M3-2-5}} \}.
\end{equation}
As \ref{M3-2L} is a relaxation of \ref{M3-2}, the objective value of the optimal solution $(x^*, y^*, \alpha^*)$ will be less than or equal to that of any feasible solutions to \ref{M3-2}. That is
\begin{equation}
    f_{cost}(x^*) =  \sum_{j\in \mathcal{J}} \sum_{i \in \mathcal{I}} (r_m h_{ij} + C_{te}q_{ij}) x_{ij}^* \leq f_{cost} (x), \; \forall (x,y) \in \mathcal{S}_2.
\end{equation}
Note $f_{cost}(x^*)$ is also the objective value of feasible solution $(x^*, y^*)$ to \ref{M3-2}.
This indicates the objective value of $(x^*, y^*)$ is no larger than any feasible solutions to \ref{M3-2}.

Therefore, $(x^*, y^*)$ is an optimal solution to \ref{M3-2} with the same objective value $f_{cost}(x^*)$, which completes the proof.
\end{proof}

\subsection{Modeling with Positional Variables}\label{sec:model-D3}

\subsubsection{Proposed Integer Optimization Model}

This model is based on the positional variables that describe the relationship between arbitrary order with the position it is assigned in the schedule. The additional decision variables are defined as follows.

\begin{description}
    \item[\quad $z_{jk}:$] $z_{jk}=1$ if order $j$ is assigned on position $k$, $\forall j, k \in \mathcal{J}$. Otherwise $z_{jk}=0$;
    \item[\quad $C_k:$] completion time of the order at position $k$, $\forall k \in \mathcal{J}$.
\end{description}

A new IO model can be formulated as below.
\begin{align}
\label{M3-3} \tag{M3-3} \quad \min \quad  
    & f_{cost}(x) \\
    \label{cons:M3-3-1}	\textrm{s.t.} \quad & \sum_{i \in \mathcal{I}} x_{ij} = 1, ~~~~~~~~~~ \forall j \in \mathcal{J},\\
    \label{cons:M3-3-2} & \sum_{k \in \mathcal{J}} z_{jk} = 1, ~~~~~~~~~~ \forall j \in \mathcal{J},\\
    \label{cons:M3-3-3} & \sum_{j \in \mathcal{J}} z_{jk} = 1, ~~~~~~~~~~ \forall k \in \mathcal{J},\\
    \label{cons:M3-3-4} & 0 \leq  C_k \leq \sum_{j \in \mathcal{J}} d_j z_{jk}, ~~~~~~~~\forall k \in \mathcal{J},\\
    \label{cons:M3-3-5} & C_1 \geq \sum_{j \in \mathcal{J}} \sum_{i \in \mathcal{I}} h_{ij} x_{ij} z_{j1}, \\
    \label{cons:M3-3-6} & C_k \geq C_{k-1} + \sum_{j \in \mathcal{J}} \sum_{i \in \mathcal{I}} h_{ij} x_{ij} z_{jk}, ~~~~~~\forall k \in \mathcal{J}/\{1\},\\
	\label{cons:M3-3-7} & x_{ij}, z_{jk} \in \{0,1\}, ~~~~~~~~~~~ \forall i \in \mathcal{I}, j,k \in \mathcal{J}.
\end{align}
Constraints \eqref{cons:M3-3-2} and \eqref{cons:M3-3-3} ensure that a particular order can only be assigned to exactly one position and each position can be assigned with exactly one order. 
Constraints \eqref{cons:M3-3-4} ensure the completion time of order at position $k$ is non-negative and no later than its due time.
Constraints \eqref{cons:M3-3-5} and \eqref{cons:M3-3-6} describe the completion time of the order at position $k$. 
The total number of variables in this formulation is $n^2 + mn +n$, which involves $n^2 + mn$ binary variables and $6n$ constraints.
Note that \ref{M3-3} is a quadratically constrained IO formulation. We will also apply linearization technique to reformulate \ref{M3-3}.

\subsubsection{Reformulation with Linearization Technique}

For arbitrary $j,~k\in \mathcal{J}$, define a new variable $\beta_{jk}=(\sum_{i\in \mathcal{I}} h_{ij} x_{ij}) z_{jk}$. An equivalent MIO model can be obtained by substituting $\sum_{i\in \mathcal{I}} h_{ij} x_{ij} z_{jk}$ in \ref{M3-3} and with the below additional constraints
\begin{equation}\label{eq:linearization-3}
    \left\{
   \begin{aligned}
   &\sum_{i \in \mathcal{I}} h_{ij} x_{ij} - M(1-z_{jk}) \leq \beta_{jk}\leq \sum_{i \in \mathcal{I}} h_{ij} x_{ij} + M(1-z_{jk}),\\
   &-M z_{jk} \leq \beta_{jk}\leq M z_{jk}.
   \end{aligned}
   \right.
\end{equation}
Note the equivalence validation of the above inequalities is the same with \eqref{eq:linearization-1}. For concise presentation we skip it.

We further reduce the above inequality constraints \eqref{eq:linearization-3} as
\begin{equation}\label{eq:linearization-4}
    \left\{
   \begin{aligned}
   & \sum_{i \in \mathcal{I}} h_{ij}x_{ij}-M(1-z_{jk}) \leq \beta_{jk},\\
   & 0\leq \beta_{jk}\leq M z_{jk}.
   \end{aligned}
   \right.
\end{equation}
Thus, a new linearized formulation \ref{M3-3L} can be obtained as:
\begin{align}
\label{M3-3L} \tag{M3-3-L} \quad \min \quad  
& f_{cost}(x) \\
    \label{cons:M3-3L-1}	\textrm{s.t.} \quad 
    & \text{Constraints } \eqref{cons:M3-3-1}-\eqref{cons:M3-3-4},\\
    \label{cons:M3-3L-3} & C_1 \geq \sum_{j \in \mathcal{J}} \beta_{j1},\\
    \label{cons:M3-3L-4} & C_k \geq C_{k-1} + \sum_{j \in \mathcal{J}} \beta_{jk}, ~~~~~~\forall k \in \mathcal{J}/\{1\},\\
    \label{cons:M3-3L-5} & \sum_{i \in \mathcal{I}} h_{ij} x_{ij} -M(1-z_{jk}) \leq \beta_{jk} \leq Mz_{jk}, ~~~~ \forall j, k \in \mathcal{J}, \\ 
    \label{cons:M3-3L-6}& \beta_{jk} \geq 0, ~~~~~~~~~~~ \forall j, k \in \mathcal{J},\\  
	\label{cons:M3-3L-7} & x_{ij}, z_{jk} \in \{0,1\}, ~~~~~~~~~~~ \forall i \in \mathcal{I}, j,k \in \mathcal{J}.
\end{align}
The value of $M$ is set as the maximum total processing time among all the orders, that is $M = \max_{i,j} \{h_{ij}\}$. 
The total number of variables is $2n^2 + mn + n$, including $n^2 + mn$ binary variables and $2n^2 + 6n$ constraints.
For formulations \ref{M3-3} and \ref{M3-3L}, we have the below Lemma \ref{lemma:linearization_equivalence2}.

\begin{lemma}\label{lemma:linearization_equivalence2}
For arbitrary optimal solution to \ref{M3-3L}, there exists an optimal solution to \ref{M3-3} with the same objective value.
\end{lemma}

\begin{proof}
The proof idea is similar to Lemma \ref{lemma:linearization_equivalence}.
For arbitrary optimal solution $(x^*, z^*, C^*, \beta^*)$ to \ref{M3-3L}, we will construct $(x^*, z^*, C^*)$ as an optimal solution to \ref{M3-3} with the same objective value.

We first show $(x^*, z^*,C^*)$ is a feasible solution to \ref{M3-3}. It is simple to check $(x^*, z^*,C^*)$ is feasible when $z_{jk}^*=0$ by substituting $z_{jk}^*=0$ into \eqref{cons:M3-3L-5} and then \eqref{cons:M3-3L-3} and \eqref{cons:M3-3L-4}. Otherwise, when substituting $z_{jk}^*=1$ we have
\begin{align}
    & \sum_{i \in \mathcal{I}} h_{ij} x_{ij}^* \leq \beta_{jk}^* \leq M, ~~~~ \forall j, k \in \mathcal{J}, \\ 
     \label{eq:lemma2-1}& C_1^* \geq \sum_{j \in \mathcal{J}} \beta_{j1}^* \geq \sum_{j \in \mathcal{J}} \sum_{i \in \mathcal{I}} h_{ij} x_{ij}^*,\\
    \label{eq:lemma2-2} & C_k^* \geq C_{k-1}^* + \sum_{j \in \mathcal{J}} \sum_{i \in \mathcal{I}} h_{ij} x_{ij}^*, ~\forall k \in \mathcal{J}/\{1\}
\end{align}
Inequalities \eqref{eq:lemma2-1} and \eqref{eq:lemma2-2} indicate $(x^*, z^*,C^*)$ is a feasible solution when $z_{jk}^*=1$.

Next, similar to Lemma \ref{lemma:linearization_equivalence}, we define the solution space of formulation \ref{M3-3} as
$
    \mathcal{S}_3 := \{ (x,z,C) : \text{Constraints \eqref{cons:M3-3-1} - \eqref{cons:M3-3-7}} \},
$
and we can get that $f_{cost}(x^*) \leq f_{cost}(x), \forall (x,z,C)\in \mathcal{S}_3$ as \ref{M3-3L} is a relaxation of \ref{M3-3}. 
Thus, we conclude $(x^*,z^*,C^*)$ is an optimal solution to \ref{M3-3}.
\end{proof}

\vspace{-0.15cm}

Based on Lemmas \ref{lemma:linearization_equivalence} and \ref{lemma:linearization_equivalence2}, we will directly solve models \ref{M3-2L} and \ref{M3-3L} using MIO solver instead of \ref{M3-2} and \ref{M3-3}, respectively.
The comparison of the three MIO formulations is discussed in next section.

\section{Computational Experiments}\label{sec:experiments}

Comprehensive computational experiments are conducted for different problem settings. For problem IDM-C, we compare the performance of methods by solving models \ref{M1} and \ref{M1-LP} with modern optimization solver \texttt{CPLEX}, DG algorithm and the expected cost method (ECM henceforth) from \cite{karandikar2021physics}. For problem IDM-CM, we compare methods by solving \ref{M2} using \texttt{CPLEX} with ECM. The three MIO models for problem IDM-CD, i.e., \ref{M3-1}, \ref{M3-2L} and \ref{M3-3L}, are solved by \texttt{CPLEX} and compared in terms of their computational efficiency.

All the methods are coded in Python. The optimization solver \texttt{CPLEX 20.1.0} is used for solving all the IO and MIO models with the Python interface \texttt{docplex}. All the settings for \texttt{CPLEX} remain default except for time limit as 600 seconds. The experiments are deployed on a mobile workstation with Intel(R) Xeon(R) W-10885M CPU @ 2.40GHz, 128 GB memory, 64 bit Windows 10 Pro operating system for workstations.

\subsection{Data Generation}\label{sec:expe_settings}

\emph{Machine tool.}
The settings for machine tool including the tool geometry, cutting parameters (except for spindle speed set) and tool life model are the same with that in \cite{karandikar2021physics}. In particular, the tool is a single-insert endmill with diameter $D=19.05$ mm and $N_t=1$. The radial depth of cut is $a=4.7$ mm, the axial depth of cut is $b=3$ mm and the feed per tooth is $f_t=0.06$ mm, respectively. 
The constants in the Taylor tool life equation are $N=0.388$ and $C=735.2$ m/min.

\emph{Order and part.}
The number of orders $n$ is from a set  $\{2,4,6,8,10,12,14,16,18,20\}$, with one single part type for each order. The quantity $l_j$ of the part for one order is an integer number randomly sampled from uniform distribution $U[10,200]$. The volume $v_j$ to be removed per part is an integer number randomly sampled from uniform distribution $U[5000,15000]$ with unit mm$^3$. The makespan limitation $c$ under each case of $n$ is decided by considering two values: the makespan of the solution by solving problem IDM-C which only considers the cost, and the minimum makespan when adopting the minimum $h_{ij}$ for each order $j$. We set $c$ as the integer number by rounding the average of these two values.

\emph{Due date.}
We specify the due date by generating the data without regarding for feasibility and then discarding instances that have no feasible schedule. Specifically, we first obtain the expected cutting speed for each order using the ECM in \cite{karandikar2021physics}. Henceforth, the makespan $\bar{C}_{max}$ and the mean value of total processing time $\bar{h}$ among the $n$ orders can be obtained. The initial due date $d_j$ is generated by randomly sampling from uniform distribution $U[0.4\bar{C}_{max}, 0.5\bar{C}_{max}]$. A schedule is then generated using earliest due date rule based on $d_j$, and the completion time $C_j$ for each order is thus obtained. We then check the feasibility of $d_j$, i.e., $d_j$ is accepted if $d_j \geq C_j$; otherwise $d_j=C_j+r$, where $r$ is a randomly sampled integer number from uniform distribution $U[0,\bar{h}]$. By this way of setting $d_j$, we ensure at least the expected speeds are feasible solution to problem IDM-CD.

\emph{Production and cost.}
The candidate cutting speed $V$ is evenly spaced in $[220,420]$ m/min with the interval 2 m/min. Thus, in total $m=100$ speeds are used for one order. 
Note here without loss of generality, we use the same set of cutting speed for all the $n$ orders.
The spindle speed matrix $\bf \Omega$ is thus constructed, with the equivalent minimum spindle speed 3676.02 rpm and the maximum spindle speed 6984.44 rpm.
Lastly, the coefficients of the cost model are set as $r_m = 2$ \$/min, $t_{ch}=2$ min and $C_{te}=\$12.5$  per tool edge.

For each order number, 10 independent data instances using different random seeds are generated with the above data generation procedure. Therefore, in total 100 instances are used. All the reported results are the average values obtained by the 10 independent instances for each order number.\\

\subsection{Results for Problem IDM-C}\label{sec:C_results}

\begin{table*}[t]
    \caption{Numerical results of solving problem IDM-C by the expected cost method (ECM) \cite{karandikar2021physics}, solving models \ref{M1} and \ref{M1-LP}, and the exact algorithm DG. }
  \label{tab:IDM_C_results}
  \resizebox{\textwidth}{!}{
  \begin{threeparttable}[t]
  \centering
  \begin{tabular}{r|rrrrrrrrrrrrrrr}
  \toprule
   \multirow{2}{*}{Order \#} && \multicolumn{2}{c}{ECM} && \multicolumn{3}{c}{M1} && \multicolumn{3}{c}{M1-LP} && \multicolumn{3}{c}{DG}  \\
  \cline{3-4} \cline{6-8} \cline{10-12} \cline{14-16}
  && Obj: Cost & Time (s) &&  Obj: Cost & CRR (\%) & Time (s) && Obj: Cost & CRR (\%) & Time (s) && Obj: Cost  & CRR (\%) & Time (s) \\
  \midrule
2	&&	2082.55	&	2.81	&&	\textbf{1977.23}	&	5.06	&	0.01	&&	\textbf{1977.23}	&	5.06	&	0.01	&&	\textbf{1977.23}	&	5.06	&	\textbf{0.000}	\\
4	&&	3301.91	&	5.58	&&	\textbf{3174.74}	&	3.85	&	0.01	&&	\textbf{3174.74}	&	3.85	&	0.02	&&	\textbf{3174.74}	&	3.85	&	\textbf{0.000}	\\
6	&&	5418.85	&	8.28	&&	\textbf{5184.44}	&	4.33	&	0.01	&&	\textbf{5184.44}	&	4.33	&	0.03	&&	\textbf{5184.44}	&	4.33	&	\textbf{0.000}	\\
8	&&	6670.49	&	11.01	&&	\textbf{6406.96}	&	3.95	&	0.01	&&	\textbf{6406.96}	&	3.95	&	0.03	&&	\textbf{6406.96}	&	3.95	&	\textbf{0.003}	\\
10	&&	9688.88	&	13.87	&&	\textbf{9314.33}	&	3.87	&	0.01	&&	\textbf{9314.33}	&	3.87	&	0.03	&&	\textbf{9314.33}	&	3.87	&	\textbf{0.003}	\\
12	&&	11872.06	&	16.66	&&	\textbf{11363.36}	&	4.28	&	0.01	&&	\textbf{11363.36}	&	4.28	&	0.05	&&	\textbf{11363.36}	&	4.28	&	\textbf{0.003}	\\
14	&&	13597.67	&	19.27	&&	\textbf{13061.32}	&	3.94	&	0.02	&&	\textbf{13061.32}	&	3.94	&	0.05	&&	\textbf{13061.32}	&	3.94	&	\textbf{0.003}	\\
16	&&	15417.88	&	21.77	&&	\textbf{14768.81}	&	4.21	&	0.02	&&	\textbf{14768.81}	&	4.21	&	0.06	&&	\textbf{14768.81}	&	4.21	&	\textbf{0.005}	\\
18	&&	16849.63	&	24.57	&&	\textbf{16127.64}	&	4.28	&	0.02	&&	\textbf{16127.64}	&	4.28	&	0.06	&&	\textbf{16127.64}	&	4.28	&	\textbf{0.003}	\\
20	&&	20654.23	&	27.29	&&	\textbf{19883.29}	&	3.73	&	0.02	&&	\textbf{19883.29}	&	3.73	&	0.08	&&	\textbf{19883.29}	&	3.73	&	\textbf{0.003}	\\
\bottomrule
\end{tabular}
\begin{tablenotes}
\item [*] All values in all columns are the average results of 10 independent runs for each order number. 
\item [*] All the 10 independent instances for all the order numbers are solved to optimality with Gap 0.00\% by \ref{M1}, \ref{M1-LP} and DG (shown by bold numbers). 
\end{tablenotes}
\end{threeparttable}
}
\end{table*}

Table \ref{tab:IDM_C_results} presents the results for solving problem IDM-C by the ECM in \cite{karandikar2021physics}, solving models \ref{M1} and \ref{M1-LP}, and the DG algorithm \ref{alg:DG}. 
The columns show the objectives (minimum cost in \$ by each method) and the solution time in second. In the column ``CRR", we calculate the cost reduction ratio (CRR) as
\begin{equation}
    CRR = \frac{Obj_{ECM} - Obj}{Obj_{ECM}} \times 100\%,
\end{equation}
where $Obj_{ECM}$ and $Obj$, respectively, refer to the objective value from ECM and other methods under comparison.

\begin{figure}[t]
    \centering
    \includegraphics[width=3.3in]{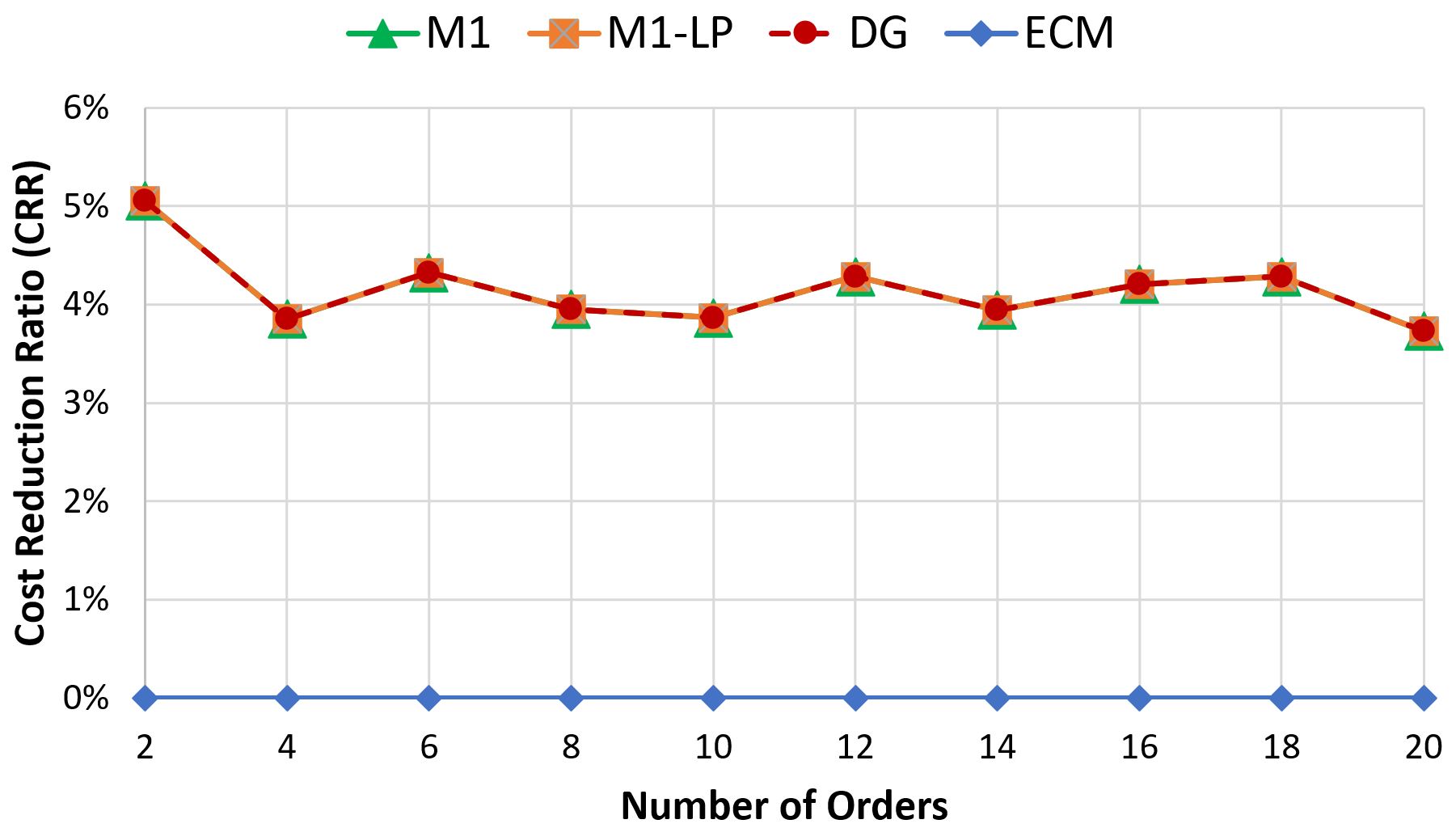}
    \caption{The comparison of cost reduction ratio (CRR) for problem IDM-C. \ref{M1}, \ref{M1-LP} and DG are overlapped as they all obtain the optimal solutions.}
    \label{fig:IDM-C_results}
\end{figure}

As shown in Table \ref{tab:IDM_C_results}, solving \ref{M1} and \ref{M1-LP} by \texttt{CPLEX}, and the DG algorithm all find the optimal solution for all the 100 instances within 0.1 second, shown as the bold numbers. Remarkably, the DG algorithm preforms as the fastest among the three, within 0.005 second for all the 100 instances. In comparison with the ECM, formulations \ref{M1} and \ref{M1-LP} and DG algorithm show that, by selecting different cutting speeds, the overall production cost can be reduced by 3\% - 5\% on average.
Figure \ref{fig:IDM-C_results} shows results of CRR in Table \ref{tab:IDM_C_results}. Note the objectives and the associated CRR for formulations \ref{M1}, \ref{M1-LP} and DG algorithm are the same and thus are overlapped.

\subsection{Results for Problem IDM-CM}\label{sec:M_results}

\begin{table*}[t]
  \centering
  \small
  \caption{Numerical results of solving problem IDM-CM by expected cost method (ECM) \cite{karandikar2021physics} and solving model \ref{M2}.}
  \label{tab:IDM_M_results}
  \resizebox{\textwidth}{!}{
  \begin{threeparttable}[t]
  \begin{tabular}{rr|rrrrrrrrrrrrr}
  \toprule
  \multirow{2}{*}{Order \#} & Makespan  && \multicolumn{3}{c}{ECM} && \multicolumn{6}{c}{M2}   \\
  \cline{4-6} \cline{8-13} 
  &Limit: $c$ &&  Obj: Cost & Time (s) & Makespan && Obj: Cost & Num & Gap (\%) & LP Relaxation & Time (s) & Makespan   \\
  \midrule
2	&	633	&&	2082.55	&	2.81	&	763.15	&&	\textbf{2052.42}	&	\textbf{10}	&	0.00	&	2034.58	&	0.00	&	624.96	\\
4	&	1027	&&	3301.91	&	5.58	&	1209.71	&&	\textbf{3245.44}	&	\textbf{10}	&	0.00	&	3233.46	&	0.05	&	1021.47	\\
6	&	1678	&&	5418.85	&	8.28	&	2000.68	&&	\textbf{5281.53}	&	\textbf{10}	&	0.00	&	5271.44	&	0.08	&	1672.02	\\
8	&	2070	&&	6670.49	&	11.01	&	2449.62	&&	\textbf{6532.52}	&	\textbf{10}	&	0.00	&	6526.56	&	0.08	&	2067.51	\\
10	&	3007	&&	9688.88	&	13.87	&	3580.07	&&	\textbf{9512.87}	&	\textbf{10}	&	0.00	&	9502.75	&	0.11	&	3003.93	\\
12	&	3669	&&	11872.06	&	16.66	&	4362.28	&&	\textbf{11571.73}	&	\textbf{10}	&	0.00	&	11566.20	&	0.08	&	3667.11	\\
14	&	4198	&&	13597.67	&	19.27	&	5013.83	&&	\textbf{13364.26}	&	\textbf{10}	&	0.00	&	13354.87	&	0.05	&	4195.26	\\
16	&	4775	&&	15417.88	&	21.77	&	5687.06	&&	\textbf{15047.90}	&	\textbf{10}	&	0.00	&	15042.45	&	0.08	&	4772.70	\\
18	&	5202	&&	16849.63	&	24.57	&	6221.69	&&	\textbf{16426.20}	&	\textbf{10}	&	0.00	&	16421.71	&	0.05	&	5199.98	\\
20	&	6413	&&	20654.23	&	27.29	&	7656.49	&&	\textbf{20286.36}	&	\textbf{10}	&	0.00	&	20281.15	&	0.08	&	6411.31	\\
\bottomrule
\end{tabular}
\begin{tablenotes}
\item [*] All values in all columns except for ``Num" are the average results of 10 independent runs for each order number. 
\item [*] Column ``Num" indicates the number of instances solved to optimality among the 10 instances.
\end{tablenotes}
\end{threeparttable}
}
\end{table*}

Table \ref{tab:IDM_M_results} shows the results for solving problem IDM-CM. In the column ``Makespan Limit: $c$", we show the average makespan limit within which all parts must be finished. Formulation \ref{M2} is solved to optimality (with the gap 0.00) for all the 10 instances under each order number within 0.11 second, shown as the bold numbers. 
The LP relaxation results show \ref{M2} is tight, which can explain for its good computational efficiency.

Notably, due to the makespan limit, the minimum cost obtained in solving problem IDM-C in Table \ref{tab:IDM_C_results} cannot be obtained in solving problem IDM-CM. Nevertheless, all these optimal solutions satisfy the makespan limit. 
While the ECM does not support to set makespan limit, the results here show that our integration provides the capability to set many practical production requirements.
Furthermore, for all the 10 order numbers, both the cost and makespan of ECM solutions are larger than those of the \ref{M2} solutions. 
This indicates by using our method, it is practical for the machine shop to achieve solutions that lead to both smaller cost and shorter makespan.

\subsection{Results for Problem IDM-CD}\label{sec:D_results}

\begin{table*}[t]
  \centering
  \small
  \caption{Numerical results of problem IDM-CD by solving models \ref{M3-1}, \ref{M3-2L} and \ref{M3-3L}.}
  \label{tab:IDM_D_results}
  \resizebox{\textwidth}{!}{
  \begin{threeparttable}[t]
  \begin{tabular}{r|rrrrrrrrrrrrrrrr}
  \toprule
  \multirow{2}{*}{Order \#} && \multicolumn{4}{c}{M3-1} && \multicolumn{4}{c}{M3-2-L} && \multicolumn{4}{c}{M3-3-L}  \\
  \cline{3-6} \cline{8-11} \cline{13-16} 
  &&  Obj: Cost  & Num & Gap (\%) & Time (s) && Obj: Cost & Num  & Gap (\%) & Time (s) && Obj: Cost & Num  & Gap (\%) & Time (s)\\
  \midrule
2	&&	\textbf{1920.68}	&	\textbf{10}	&	0.00	&	0.02	&&	\textbf{1920.68}	&	\textbf{10}	&	0.00	&	0.02	&&	\textbf{1920.68}	&	\textbf{10}	&	0.00	&	0.03	\\
4	&&	\textbf{4463.55}	&	\textbf{10}	&	0.00	&	0.08	&&	\textbf{4463.55}	&	\textbf{10}	&	0.00	&	0.05	&&	\textbf{4463.55}	&	\textbf{10}	&	0.00	&	0.07	\\
6	&&	\textbf{5412.14}	&	\textbf{10}	&	0.00	&	0.25	&&	\textbf{5412.14}	&	\textbf{10}	&	0.00	&	0.07	&&	\textbf{5412.14}	&	\textbf{10}	&	0.00	&	0.14	\\
8	&&	\textbf{7293.43}	&	\textbf{10}	&	0.00	&	0.27	&&	\textbf{7293.43}	&	\textbf{10}	&	0.00	&	0.10	&&	\textbf{7293.43}	&	\textbf{10}	&	0.00	&	1.83	\\
10	&&	\textbf{11129.40}	&	\textbf{10}	&	0.00	&	1.02	&&	\textbf{11129.40}	&	\textbf{10}	&	0.00	&	0.39	&&	\textbf{11129.40}	&	\textbf{10}	&	0.01	&	25.43	\\
12	&&	\textbf{12700.74}	&	\textbf{10}	&	0.01	&	14.17	&&	\textbf{12700.74}	&	\textbf{10}	&	0.00	&	2.66	&&	12700.74	&	5	&	3.14	&	508.27	\\
14	&&	12615.67	&	8	&	0.44	&	163.11	&&	\textbf{12615.67}	&	\textbf{10}	&	0.01	&	10.80	&&	12615.67	&	1	&	6.05	&	601.27	\\
16	&&	13232.52	&	4	&	2.52	&	485.71	&&	\textbf{13232.52}	&	\textbf{10}	&	0.01	&	45.09	&&	13232.52	&	1	&	6.06	&	602.53	\\
18	&&	18001.38	&	0	&	6.47	&	600.30	&&	\textbf{18001.38}	&	\textbf{10}	&	0.01	&	184.39	&&	18001.38	&	0	&	8.06	&	602.33	\\
20	&&	21394.91	&	0	&	8.76	&	600.60	&&	\textbf{21394.91}	&	\textbf{8}	&	0.76	&	380.35	&&	21394.91	&	0	&	8.76	&	602.19	\\
\bottomrule
\end{tabular}
\begin{tablenotes}
\item [*] All values in all columns except for ``Num" are the average results of 10 independent runs for each order number. 
\item [*] Column ``Num" indicates the number of instances solved to optimality among the 10 instances.
\end{tablenotes}
\end{threeparttable}
}
\end{table*}

Table \ref{tab:IDM_D_results} presents the numerical results for solving problem IDM-CD with formulations \ref{M3-1}, \ref{M3-2L} and \ref{M3-3L}, in terms of the objective value, number of the instances solved to optimality out of 10, gap to the lower bound returned by \texttt{CPLEX} and the solution time.
Obviously, the solution quality of \ref{M3-2L} is the better than that of \ref{M3-1} and further better than that of \ref{M3-3L}.
This can be seen from the number of instances solved to optimality.
Among all the 100 instances, \ref{M3-2L} solves 98 to optimality except for 2 instances at $n=20$ orders, while that of \ref{M3-1} and \ref{M3-3L} are 72 and 66, respectively.
The same conclusions can be drawn from the gap and the solution time. 
For example, at $n=18$, \ref{M3-2L} takes 0.01 second to solve all the 10 instances to optimality, while \ref{M3-1} takes 600.30 seconds yet still has gap 6.47\%, and \ref{M3-3L} takes 602.33 seconds with gap 8.06\%.
Note that time above 600 seconds means the solving procedure is terminated due to time limit of \texttt{CPLEX}.

It is interesting to see although with different gaps for the three models, the objectives are all the same.
In our experience, it can often be the case that \texttt{CPLEX} has found the optimal solution in 5 minutes, though it takes much longer time to certify its optimality. Thus, formulations with fast increasing lower bounds can certify the optimality in shorter time.

As a closing remark for this section, the results with cost, makespan and due date consideration validate the capability our methods to integrate constraints for describing practical requirements for the operational excellence of machine tool.

\section{Conclusion, Discussion and Future Work}\label{sec:conclusion}

In summary, we have identified and studied a significant problem of the integration of discrete-event dynamics and machining dynamics for the operational excellence of machine tool for the first time. 
The machining stability, various cutting speeds and tool life were incorporated into the production of a series of parts from multiple orders in the queue of machine tool.
Based on a state-of-the-art logistic regression model for tool life prediction, a new cost function was developed for orders with cutting speed choices as the variables.
Then the cost was minimized via choosing different stable cutting speeds.

We have demonstrated a set of models and algorithms by considering various practical production scenarios, including cost, makespan and due date.
When considering the integrated optimization of cost and makespan, we developed integer optimization models and exact algorithm. Then, we considered the integrated optimization of cost and due date. Three categories of mixed-integer optimization models were presented based on the choices of decision variables. The linearization technique was applied to obtain the models.
The results showed while satisfying the makespan and due date requirements, our methods achieved on average 3\% - 5\% reduction of the cost, which has the potential for cost saving in practice.

Future work includes the deeper integration from both the discrete-event dynamics and the machining dynamics. Besides cutting speed, other machining parameters can be as decision variables, including axial depth of cut and feed rate. Thus, we may have nonlinear and nonconvex constraints from the stability lobe diagram. Also, the stochastic tool life prediction can be considered.
In the meantime, the complexities of the problems for makespan and due date constraints are one future direction. The practical extensions of the problems are to consider different candidate speeds for orders, various part types of an order, and multiple machine tools in the machine shop.

\section*{Acknowledgements}

This work has been supported in part by the University of Tennessee Knoxville under the Graduate Advancement Training and Education (GATE) program of Science Alliance, and the Southeastern Advanced Machine Tools Networks (SEAMTN).

\bibliographystyle{abbrv}
\bibliography{mybibfile}

\end{document}